\documentclass[12pt]{amsart}
\usepackage{enumerate,amsmath,amssymb,bm}
\usepackage{url}
\usepackage{xspace}

\usepackage[margin=1in]{geometry}

\DeclareMathOperator{\pnt}{\raise 0.5mm \hbox{\large\textbf{.}}}

\newcommand{\note}[2][ ]{}

\newtheorem{theorem}{Theorem}
\newtheorem{lemma}[theorem]{Lemma}

\newtheorem{conjecture}[theorem]{Conjecture}

\theoremstyle{definition}
\newtheorem{remark}[theorem]{Remark}

\usepackage[pdfauthor={Fabrizio Zanello},
            pdftitle={Parity},
            pdfsubject={enumerative combinatorics},
            pdfkeywords={parity},
            pdfproducer={Latex with hyperref},
            pdfcreator={latex->dvips->ps2pdf},
            pdfpagemode=UseNone,
            bookmarksopen=false,
            bookmarksnumbered=true]{hyperref}

\begin{document}
\title[The positive odd density of $p(n)$ from that of a multipartition function]{Deducing the positive odd density of $p(n)$ from that of a multipartition function: An unconditional proof}
\author{Fabrizio Zanello} \address{Department of Mathematical Sciences\\ Michigan Tech\\ Houghton, MI  49931-1295}
\email{zanello@mtu.edu}

\maketitle

\begin{abstract} 
A famous conjecture of Parkin-Shanks predicts that $p(n)$ is odd with density $1/2$. Despite the remarkable amount of work of the last several decades, however, even showing this density is positive seems out of reach. In a 2018 paper with Judge, we introduced a different approach and conjectured the \lq \lq striking'' fact that, if for any $A \equiv \pm 1\ (\bmod\ 6)$ the multipartition function $p_A(n)$ has positive odd density, then so does $p(n)$. Similarly, the positive odd density of any $p_{A}(n)$ with $A\equiv 3\ (\bmod\ 6)$ would imply that of $p_3(n)$.

Our conjecture was shown to be a corollary of an earlier conjecture of the same paper. In this brief note, we provide an unconditional proof of it. An important tool will be Chen's recent breakthrough on a special case of our earlier conjecture. 
\end{abstract}
{\ }\\
\\
One of the most fascinating and intractable problems in partition theory is study of the parity of the partition function $p(n)$. In particular, a classical conjecture of Parkin-Shanks \cite{PaSh} predicts that $p(n)$ has odd density $1/2$ (see also \cite{Calk}). Despite a large amount of literature devoted to this problem, however, even showing that the odd density of $p(n)$ exists and is positive still appears out of reach. The best result currently available only guarantees that the number of odd values of $p(n)$ for $n\leq x$ has order at least $\frac{\sqrt{x}}{\log\log x}$, for $x$ large \cite{BGS}. This bound can be extended, for any  $t\ge 1$, to the \emph{$t$-multipartition function} $p_t(n)$, defined by
$$\sum_{n=0}^{\infty} p_t(n)q^n = \frac{1}{\prod_{i=1}^{\infty} (1-q^i)^t}.$$
Note that $p_1(n)=p(n)$.

Recall that the \emph{odd density} (or {density} of the odd values) of $p_t(n)$ is
$$\delta_t=\lim_{x\to \infty}\frac{\#\{n\leq x: p_t(n)\ \text{is odd}\}}{x},$$
if this limit exists. In particular, $\delta_1$ is the odd density of $p(n)$, while $\delta_3$ is the odd density of the cubic partition function. Equivalently, assuming they both exist, $\delta_3$ is precisely 8 times the odd density of the Fourier coefficients of the Klein $j$-function \cite{Za}. For obvious parity reasons, since $\delta_{2^c\cdot m}=\delta_m/2^c$, it suffices to restrict our attention to $\delta_t$ for $t$ odd. 

In \cite{jkza}, joint with Judge and Keith, we conjectured that $\delta_t =1/2$ for all odd values of $t$. Unfortunately, similarly to the case $t=1$, proving that $\delta_t$ exists and is positive appears extremely difficult, for any $t$. 

In \cite{JZ}, joint with Judge, we introduced a new approach to the study of the parity of $p(n)$. In Conjecture 2.4, we predicted the existence of a doubly-indexed, infinite family of identities modulo 2, which suitably related $p(n)$ to the other multipartition functions. The main conjecture of \cite{JZ}, itself a corollary to 2.4, is the \lq \lq striking'' fact that, under reasonable existence assumptions, the positive odd density of $p(n)$ follows from that of $p_A(n)$, for any $A \equiv \pm 1\ (\bmod\ 6)$; similarly, if $\delta_{A}>0$ for any $A\equiv 3\ (\bmod\ 6)$, then $\delta_{3}>0$.

Even though \cite{JZ}, Conjecture 2.4 is in general still open, the goal of this note is to establish its corollary unconditionally. A key ingredient will be a recent breakthrough by Chen \cite{Ch}, who proved an important infinite case of Conjecture 2.4.

We begin with the statement of 2.4. (\cite{JZ}, Conjecture 2.3 corresponds to the special case $t=1$ and will not be restated here.) Unless otherwise noted, all congruences are modulo 2.

\begin{conjecture}[\cite{JZ}, Conjecture 2.4]\label{2.4} 
Fix odd positive integers $a$ and $t$, where $t\equiv 3\ (\bmod\ 6)$ if $a \equiv 3\ (\bmod\ 6)$. Let $k=\left\lceil \frac{t(a^2-1)}{24a}\right\rceil$. Then
\begin{equation}\label{CongOfForm2}
q^k \sum_{n=0}^{\infty} p_t(an+b)q^n \equiv \sum_{d|a} \sum_{j=0}^{\lfloor k/d \rfloor} \frac{\epsilon^t_{a,d,j}\ q^{dj}}{\prod_{i\geq 1} (1-q^{di})^{\frac{at}{d}-24j}},
\end{equation}
where \[
    b= 
\begin{cases}
   \ \ \ 0, &\text{if }\ a=1;\\
    \frac{t}{3}\cdot8^{-1}\ (\bmod\ a),& \text{if }\ t\equiv 3\ (\bmod\ 6);\\
    t\cdot24^{-1}\ (\bmod\ a),              & \text{otherwise},
\end{cases}
\]
{\ }\\
for a suitable choice of the $\epsilon^t_{a,d,j}\in\{0,1\}$, with $\epsilon^t_{a,1,0}=1$ and $\epsilon^t_{a,d,j}=0$ if $at/d-24j<0$.
\end{conjecture}

The following is Chen's result from \cite{Ch}.

\begin{lemma}[\cite{Ch}, Theorem 1.6]\label{chen}
Conjecture \ref{2.4} holds for:
\begin{itemize}
\item[i)] $a=p^{\alpha}$, with $p\ge 5$ prime and $\alpha \ge 1$, and any $t\ge 1$ odd; 
\item[ii)]  $a=3$ and any $t\ge 3$, $t \equiv 3\ (\bmod\ 6)$.
\end{itemize}
\end{lemma}

The main conjecture of \cite{JZ} can be stated as follows.
\begin{conjecture}[\cite{JZ}, Corollaries to Conjectures 2.3 and 2.4]\label{delta}
{\ }
\begin{itemize}
\item[i)] Suppose there exists an integer $A \equiv \pm 1\ (\bmod\ 6)$ such that $\delta_{A}>0$, and assume $\delta_i$ exists for all $i\le A$, $i\equiv \pm 1\ (\bmod\ 6)$. Then $\delta_1>0$.
\item[ii)] Suppose there exists an integer $A \equiv 3\ (\bmod\ 6)$ such that $\delta_{A}>0$, and assume $\delta_i$ exists for all $i\le A$, $i \equiv 3\ (\bmod\ 6)$. Then $\delta_3>0$.
\end{itemize}
\end{conjecture}

We only remark here that,  while Conjecture \ref{2.4} is stronger than Conjecture \ref{delta}, showing Conjecture \ref{2.4} for \emph{specific} values of $a$ does not in general imply Conjecture \ref{delta} for the same values of $A$. However, using Lemma \ref{chen} for $a$ prime along with a careful inductive argument, we can now provide an unconditional proof of Conjecture \ref{delta}, for all $A$.

\begin{theorem}\label{main}
Conjecture \ref{delta} is true.
\end{theorem}

\begin{proof}
We start with i), and proceed by induction on $A\ge 5$. The case $A=5$ is already known (see \cite{jkza}, Theorem 2). Thus, assume the result holds up to $A-2$ (or $A-4$, depending on the value of $A$ modulo 3), and let $\delta_A>0$. 

Let $p$ be any prime dividing $A$. Since $p\ge 5$, Lemma \ref{chen}, i) guarantees that Conjecture \ref{2.4} has a solution for $a=p$ and $t=A/p$. That is, we have an identity modulo 2 of the form

\begin{equation}\label{p}
q^k \sum_{n=0}^{\infty} p_{A/p}(pn+b)q^n \equiv \sum_{d=1,p} \sum_{j=0}^{\lfloor k/d \rfloor} \frac{\epsilon_{d,j}\ q^{dj}}{\prod_{i\geq 1} (1-q^{di})^{\frac{A}{d}-24j}},
\end{equation}  
{\ }\\
for $b$ and $k$ determined by $p$ and $A$ as from the statement of Conjecture \ref{2.4}, all $\epsilon_{d,j}=0$ or $1$, and $\epsilon_{1,0}=1$. 

Notice that the condition $\epsilon_{1,0}=1$ implies the existence of the summand
$$\frac{1}{\prod_{i\geq 1} (1-q^{i})^A}$$
on the right side of (\ref{p}). Further, all other nonzero summands on the right side are easily seen to be of the form 
$$\frac{q^{dj}}{\prod_{i\geq 1} (1-q^{di})^B},$$
where $0<B<A$, $B \equiv \pm 1\ (\bmod\ 6)$.

Hence, if any corresponding $\delta_B>0$, by induction we are done. Otherwise, assume all $\delta_B=0$, and as a consequence, note that the number of odd coefficients up to degree $x$ on the right side of (\ref{p}) is given by
$$\delta_A \cdot x +o(x),$$
for $x$ large. Therefore, the odd coefficients  on the right side of (\ref{p}) have density $\delta_A$, and the same obviously holds true for the left side.

Since $p_{A/p}(pn+b)$ denotes the $(A/p)$-multipartition function along the arithmetic progression $pn+b$, it is clear that if this latter has odd density $\delta_A$, then the multipartition function $p_{A/p}(n)$ itself has odd density
$$\delta_{A/p} \geq \frac{1}{p}\cdot \delta_A >0.$$
Since $A/p <A$, $A/p \equiv \pm 1\ (\bmod\ 6)$, the inductive hypothesis again gives $\delta_1>0$, as desired.

The proof of part ii) is similar, so we will only sketch it. The base case to run the induction on $A \equiv 3\ (\bmod\ 6)$ is that $\delta_9>0$  implies $\delta_3>0$, under the usual existence assumptions on the $\delta_i$. This was again shown in \cite{jkza}, Theorem 2. 

Next suppose that $\delta_A>0$, and let $p$ be the largest prime that divides $A$. Set $a=p$ and $t=A/p$. Note that, here, $p=3$ precisely when $A\ge 9$ is a power of 3, so we always have $t\equiv 3\ (\bmod\ 6)$. Therefore, using both parts of Lemma \ref{chen}, we are again guaranteed the existence of an identity of the form (\ref{p}).

Now notice that, on the right side of (\ref{p}), the exponent $$\frac{A}{d}-24j$$ is always positive, congruent to 3 $(\bmod\ 6)$, and except when $d=1$ and $j=0$, it is smaller than $A$. The rest of the argument to show that $\delta_3>0$ is identical to part i). This completes the proof of the theorem.
\end{proof}

\begin{remark}
Recall that, even though Conjecture \ref{delta} has been proven, Conjecture \ref{2.4} remains open for most values of $a$. We believe these identities modulo 2 to be of significant independent interest, so we hope Chen's result from \cite{Ch} (most likely combined with a new idea) can be extended to show Conjecture \ref{2.4} in full.

Finally, despite multiple attempts, we have not been able to relate the conjectural positive density of $p(n)$ to that of $p_3(n)$. One issue in this sense is that no identity modulo 2, similar to those of Conjecture \ref{2.4} but simultaneously involving $p(n)$ and $p_3(n)$, appears to exist. Thanks to Theorem \ref{main}, establishing that $\delta_3>0$ implies $\delta_1>0$ would immediately give us, under standard existence assumptions on the $\delta_i$, that $p(n)$ has positive odd density whenever \emph{any} multipartition function $p_A(n)$ does.
\end{remark}

\section*{Acknowledgements} 
We warmly thank Steven J. Miller for discussions on the submission. We are also grateful to Samuel Judge and Shi-Chao Chen; without their important contributions to this line of research, our result would have not been possible. This work was partially supported by a Simons Foundation grant (\#630401).

\end{document}